\documentclass[11pt]{article}
\usepackage{amssymb}
\usepackage{amsfonts}
\usepackage{amsmath}
\usepackage{amscd}
\usepackage{amsthm}
\usepackage{latexsym}
\usepackage[all]{xy}
\usepackage{CJK}
\usepackage{setspace}
\usepackage{indentfirst}
\numberwithin{equation}{section}
\date{}

\newtheorem*{theorem1}{Theorem A}
\newtheorem*{theorem2}{Theorem B}
\newtheorem*{theorem3}{Theorem C}
\newtheorem{thm}{Theorem}[section]
\newtheorem{prop}[thm]{Proposition}
\newtheorem{lem}[thm]{Lemma}

\begin{document}
\title{\bf A new characterization of Auslander algebras$^\star$}
\author{{\small Shen Li and  Shunhua Zhang}\\
         {\small School of Mathematics, Shandong University,
        Jinan, 250100,P.R.China}}
\date{}
\maketitle

\begin{center}\section*{}\end{center}

\begin{abstract}
Let $\Lambda$ be a finite dimensional Auslander algebra. For a $\Lambda$-module $M$, we prove that the projective dimension of $M$ is at most one if and only if the projective dimension of its socle soc\,$M$ is at most one. As an application, we give a new characterization of Auslander algebra $\Lambda$, and prove that a finite dimensional  algebra $\Lambda$ is an Auslander algebra provided its global dimension gl.d\,$\Lambda\leq2$ and an injective $\Lambda$-module is projective if and only if the projective dimension of its socle is at most one.
\end{abstract}

{\bf Key words and phrases:}\ Auslander algebra, projective dimension, global dimension.

\footnote {MSC(2000): 16E10, 16G10.}

\footnote{ $^\star$ Supported by the NSF of China (Grant No.11171183  and  11371165),  and also supported by PCSIRT ( IRT1264).}

\footnote{  Email addresses: \  fbljs603@163.com(S.Li), \ \    shzhang@sdu.edu.cn(S.Zhang).}

\vskip0.2in

\section{Introduction}
A finite dimensional algebra $\Lambda$ is called an Auslander algebra if its global dimension is at most two and dominant dimension is at least two, that is, in the minimal injective resolution $0\rightarrow \Lambda \rightarrow I_{0}\rightarrow I_{1}\rightarrow I_{2}\rightarrow 0 $ of $\Lambda$, both $I_{0}$ and $I_{1}$ are projective $\Lambda$-modules. It is introduced by Auslander when studying representation-finite algebras in [1] and can be constructed in the following way: let $R$ be a finite dimensional algebra of finite representation type and $M_{1},M_{2},\ldots,M_{n}$ be a complete set of representatives of the isomorphism classes of indecomposable $R$-modules, then $\Lambda=$End$_{R}(\oplus^{n}_{i=1}M_{i})$ is the Auslander algebra of $R$. Moreover, this construction induces a mutually inverse bijection between Morita equivalence classes of representations-finite algebras and Morita equivalence classes of Auslander algebras. This bijection is called the Auslander correspondence and the correspondence is generalized by Iyama to a higher dimensional version in [2]. It is also known that Auslander algebras have close relations with quasi-hereditary algebras, preprojective algebras and projective quotient algebras, see [3,4,5] for details.

Let $M$=$\oplus^{n}_{i=1}M_{i}$ be the additive generator of $R$ and $\Lambda$$=$End$_{R}M$ be the corresponding Auslander algebra. We denote by $S_{i}$ the simple top of the indecomposable projective $\Lambda$-module Hom$_{R}$($M,M_{i}$), then the projective dimension pd$_{\Lambda}\,S_{i}$ of $S_{i}$ is at most two since the global dimension gl.d\,$\Lambda$ of $\Lambda$ is at most two. According to [1,VI, Proposition 5.11], we know that pd$_{\Lambda}\,S_{i}$$\leq 1$ if $M_{i}$ is a projective $R$-module and otherwise pd$_{\Lambda}\,S_{i}$$=2$. Thus the projective dimensions of all simple $\Lambda$-modules are clear. However, for a non-simple $\Lambda$-module $M$, we don't know whether its projective dimension pd$_{\Lambda}\,M$ is one or two. In this paper, we show that the projective dimension pd$_{\Lambda}\,M$ of $M$ is determined by the projective dimension of its socle soc$M$. Since soc$M$ is the direct sum of the simple submodules of $M$, by the above discussion, we can easily calculate the projective dimension of soc$M$.

 Note that the socle of $M$ coincides with the socle of its injective envelope $I$($M$). Thus in order to investigate the relations between the projective dimensions of $M$ and soc\,$M$, we should study injective $\Lambda$-modules and their socles at first. Let $\mathcal{C}$($\Lambda$) be the full subcategory of mod-$\Lambda$ whose objects are all the projective-injective $\Lambda$-modules, then we prove the following result about injective $\Lambda$-modules.

\begin{theorem1}

Let $\Lambda$ be a finite dimensional Auslander algebra and $I$ be an indecomposable injective $\Lambda$-module. Then $I$ is projective if and only if the projective dimension of its socle is at most one, that is,  $\mathcal{C}(\Lambda)$$=\{ I \in mod$-$\Lambda\ |\  I\ is\ injective \ and \ pd_{\Lambda}\,soc\,I\leq 1\}$.

\end{theorem1}

Now we show the relations between the projective dimensions of $\Lambda$-modules and their socles. Let $\mathcal{P}^{1}$($\Lambda$) be the full subcategory of mod-$\Lambda$ whose objects are all the $\Lambda$-modules $M$ with pd$_{\Lambda}\,M\leq 1$,  we give the second result of this paper.

\begin{theorem2}

Let $\Lambda$ be a finite dimensional Auslander algebra and $M$ be a $\Lambda$-module. Then the projective dimension of $M$ is at most one if and only if the projective dimension of its socle soc\,$M$ is at most one, that is,  $\mathcal{P}^{1}(\Lambda)=\{M\in mod$-$\Lambda\ | \ pd_{\Lambda}\,soc\,M \leq 1\}$.

\end{theorem2}

Theorem B also implies that pd$_{\Lambda}\,M=2$ if and only if pd$_{\Lambda}\,{\rm soc}\,M=2$. Before we complete this paper, Eir\'{\i}ksson gives a characterization of $\mathcal{P}^{1}(\Lambda)$ in [7], which states that $\mathcal{P}^{1}(\Lambda)$ consists of all $\Lambda$-modules cogenerated by projective $\Lambda$-modules. We investigate $\mathcal{P}^{1}(\Lambda)$ from a different point of view and stress that the projective dimension of a $\Lambda$-module $M$ is completely determined by the projective dimension of its socle.

The above two theorems state the properties of Auslander algebras, then it is natural to ask whether Auslander algebras can be characterized by these properties, that is, whether a finite dimensional algebra of global dimension at most two is an Auslander algebra provided it satisfies the properties in Theorem A or Theorem B. We give a positive answer to this question as following.

\begin{theorem3}

Let $\Lambda$ be a finite dimensional algebra. Then $\Lambda$ is an Auslander algebra if and only if its global dimension gl.d\,$\Lambda\leq 2$ and $\mathcal{C}(\Lambda)$ $=\{ I \in mod$-$\Lambda\ |\  I \ is \ injective  \  and  \\  pd_{\Lambda} \, soc\, I  \leq 1 \}$.

\end{theorem3}

We also provide an example to show that a finite dimensional algebra $\Lambda$ with gl.d\,$\Lambda\leq 2$ and  $\mathcal{P}^{1}(\Lambda)=\{M\in {\rm mod}$-$\Lambda\ | \ {\rm pd_{\Lambda}\,soc}\,M \leq 1\}$ is not necessarily an Auslander algebra.

This paper is arranged as follows. In section 2, we fix the notions and recall some necessary facts needed for our research. In section 3, we prove Theorem A and B. Section 4 is devoted to the proof of Theorem C.

\section{Preliminaries}

Throughout this paper, let $k$ be an algebraically closed field and we consider basic finite dimensional $k$-algebras. For a finite dimensional $k$-algebra $A$, we denote by mod-$A$ the category of finitely generated right $A$-modules and by gl.d\,${A}$ the global dimension of $A$. For a right $A$-module $M$, pd$_{A}\,M$ (Id$_{A}\,M$) is the projective (injective) dimension of $M$ and soc\,$M$(rad\,$M$) is the socle (radical) of $M$. We denote by add\,$M$ the full subcategory of mod-$A$ whose objects are direct summands of finite direct sums of copies of $M$. $\tau_{A}$ is the Auslander-Reiten translation of $A$.

Let $R$ be a finite dimensional $k$-algebra of finite representation type and $M_{1},M_{2},\ldots,M_{n}$ be a complete set of representatives of the isomorphism classes of indecomposable $R$-modules. Then $M$=$\oplus^{n}_{i=1}M_{i}$ is the additive generator of $R$ and $\Lambda$$=$End$_{R}M$ is the corresponding Auslander algebra. For a $R$-module $X$, we denote by $P_{X}$$=$Hom$_{R}$($M,X$) the corresponding projective $\Lambda$-module. Now we recall some basic properties of Auslander algebras.

\begin{prop} {\rm [6,Proposition 2.3]} Let $S_{i}$ $(i=1,2,\ldots,n)$ be the simple top of the indecomposable projective $\Lambda$-module $P_{M_{i}}$$=$Hom$_{R}(M,M_{i})$. Then we have

\par $(1)$ pd$_{\Lambda}\,S_{i}\leq 1$ if and only if $M_{i}$ is projective. Then $0 \rightarrow P_{radM_{i}}\rightarrow P_{M_{i}}\rightarrow S_{i}\rightarrow 0 $ is a minimal projective resolution of $S_{i}$.

\par $(2)$ pd$_{\Lambda}\,S_{i}= 2 $ if and only if $M_{i}$ is non-projective. Then the almost split sequence $0 \rightarrow \tau M_{i}\rightarrow E \rightarrow M_{i}\rightarrow 0$ gives a minimal projective resolution $0 \rightarrow P_{\tau M_{i}} \rightarrow P_{E} \rightarrow P_{M_{i}} \rightarrow S_{i} \rightarrow 0$ of $S_{i}$.

\end{prop}

The above proposition shows the relationship between almost split sequences in mod-$R$ and projective resolutions of simple $\Lambda$-modules.

\begin{lem}{\rm [6,Lemma 2.4]} Let $R$ and $\Lambda$ be as above and $M_{i}$ be a non-projective $R$-module. Then we have Ext\,$^{2}_{\Lambda}(S_{i},\Lambda)\neq 0$ and Ext\,$^{j}_{\Lambda}(S_{i},\Lambda)= 0$  if $j\neq2$.

\end{lem}

If $j=0$, we get Hom$_{\Lambda}$($S_{i}, \Lambda$)=0 and this implies that the socle of $\Lambda$ is the direct sum of simple $\Lambda$-modules whose projective dimensions are at most one. Now we give a general result about algebras of global dimension two.

\begin{lem}

Let $A$ be a finite dimensional $k$-algebra with gl.d\,$A=2$. Then we have pd$_{A}\,soc\,A\leq 1$.

\end{lem}

\begin{proof}
Obviously we can get a short exact sequence $0\rightarrow $soc$ \,A\rightarrow A \rightarrow A/$soc$A \rightarrow 0$. If $A/$soc$A$ is a projective $A$-module, then this sequence splits and soc\,$A$ is also projective. Otherwise pd$_{A}\,$soc$\,A$=pd$_{A}(A/$soc$A )-1\leq 1$ since $A$ is projective and  ${\rm gl.d}\ A=2$.

\end{proof}

We also need the following lemma in this paper.

\begin{lem}

{\rm [1,VI, Lemma 5.5]} Let $A$ be a finite dimensional $k$-algebra and gl.d\,$A=n$. Then we have Id$_{A}\,A=n$.

\end{lem}

The following proposition shows the relations between the projective dimensions of the three modules in a short exact sequence, which is very useful.

\begin{prop}

{\rm [8,Appendix, Proposition 4.7]} Let $A$ be a finite dimensional $k$-algebra and $0 \rightarrow L \rightarrow M \rightarrow N \rightarrow 0$ be a short exact sequence of $A$-modules. Then we have
\par $(1)$ pd$_{A}\,N\leq max(pd_{A}\,M, 1+pd_{A}\,L)$, and the equality holds if $pd_{A}\,M\neq pd_{A}\,L$.
\par $(2)$ pd$_{A}\,L\leq max(pd_{A}\,M,-1+pd_{A}\,N)$, and the equality holds if $pd_{A}\,M\neq pd_{A}\,N$.
\par $(3)$ pd$_{A}\,M\leq max(pd_{A}\,L, pd_{A}\,N)$, and the equality holds if $pd_{A}\,N\neq 1+pd_{A}\,L$.

\end{prop}

Throughout this paper, we follow the standard terminologies and notations used in the representation theory of algebras, see [1,8].

\section{Projective dimensions of modules over Auslander algebras}

Let $\Lambda$ be a finite dimensional Auslander $k$-algebra. In this section, we investigate the $\Lambda$-modules with projective dimension at most one and show that they are determined by the projective dimension of their socles.

By Lemma 2.3, we know that the socle of $\Lambda$ is the direct sum of simple $\Lambda$-modules with projective dimension at most one. Furthermore, we prove that all simple $\Lambda$-modules with projective dimension at most one are contained in the socle of $\Lambda$.

\begin{prop}

Let $\Lambda$ be a finite dimensional Auslander $k$-algebra and $S_{i}$ be a simple $\Lambda$-module. Then we have pd$_{\Lambda}\,S_{i}\leq 1$ if and only if $S_{i}\in add\,soc\,\Lambda$.

\end{prop}

\begin{proof}

If $S_{i}\in $add$\,$soc$\,\Lambda$, then pd$_{\Lambda}\,S_{i}\leq 1$ since pd$_{\Lambda}$soc$\,\Lambda \leq 1$ by Lemma 2.2.

Now assume pd$_{\Lambda}\,S_{i}\leq 1$, we only need to show Hom$_{\Lambda}(S_{i},\Lambda)\neq 0$. According to Proposition 2.1, there exists an indecomposable projective $R$-module $M_{i}$ such that  $0 \rightarrow $Hom$_{R}(M,{\rm rad}M_{i})\rightarrow $Hom$_{R}(M,M_{i})\rightarrow S_{i}\rightarrow 0 $ is a minimal projective resolution of $S_{i}$. Applying Hom$_{\Lambda}(-,\Lambda)$ to the above short exact sequence, we get an exact sequence
$$0 \rightarrow {\rm Hom}_{\Lambda}(S_{i}, \Lambda)\rightarrow {\rm Hom}_{\Lambda}({\rm Hom}_{R}(M,M_{i}),\Lambda)\rightarrow {\rm Hom}_{\Lambda}({\rm Hom}_{R}(M,{\rm rad}M_{i}),\Lambda)$$  Note that $\Lambda$$=$End$_{R}M$ and Hom$_{R}(M,-)$ induces an equivalence between add\,$M$ and add\,$\Lambda$, then we get the following exact sequence $$0 \rightarrow {\rm Hom}_{\Lambda}(S_{i},\Lambda)\rightarrow {\rm Hom}_{R}(M_{i},M) \xrightarrow{g^{*}} {\rm Hom}_{R}({\rm rad}M_{i},M)$$

Let $\varphi$ be the epimorphism from $M_{i}$ to top\,$M_{i}$ and we decompose $M={\rm top}\,M_{i}\oplus M^{'}$, then $(\varphi, 0)^{t}\in$ Hom$_{R}(M_{i},M)$. Obviously $g^{*}((\varphi, 0)^{t})=g^{*}(0)=0$ and $g^{*}$ is not injective. This implies that Hom$_{\Lambda}(S_{i},\Lambda)\neq 0$.

\end{proof}

It is known that the socle of a $\Lambda$-module $M$ coincides with the socle of its injective envelope $I(M)$, so we investigate the relationship between injective $\Lambda$-modules and their socles at first.

\begin{thm}

Let $\Lambda$ be a finite dimensional Auslander $k$-algebra and $I$ be an indecomposable injective $\Lambda$-module. Then $I$ is projective if and only if the projective dimension of its socle is at most one, that is,  $\mathcal{C}(\Lambda)$$=\{ I \in mod$-$\Lambda\ |\  I\ is\ injective \ and \ pd_{\Lambda}\,soc\,I\leq 1\}$.

\end{thm}

\begin{proof}

If $I$ is a projective $\Lambda$-module, then soc\,$I\in $add$\,{\rm soc}\,\Lambda$ and by Proposition 3.1, pd$_{\Lambda}\,{\rm soc}\,I\leq 1$.

Conversely, assume pd$_{\Lambda}\,{\rm soc}\,I\leq 1$. Again by Proposition 3.1, we have soc$\,I\in $add$\,{\rm soc}\,\Lambda$. This implies that $I$ is a direct summand of the injective envelope $I(\Lambda)$ of $\Lambda$. Since the dominant dimension of $\Lambda$ is at least two, $I(\Lambda)$ is a projective $\Lambda$-module. Thus
$I$ is also projective.

\end{proof}

Now we are in a position to prove the main result of this section.

\begin{thm}

Let $\Lambda$ be a finite dimensional Auslander $k$-algebra  and $M$ be a $\Lambda$-module. Then the projective dimension of $M$ is at most one if and only if the projective dimension of its socle soc\,$M$ is at most one, that is,  $\mathcal{P}^{1}(\Lambda)=\{M\in mod$-$\Lambda\ | \ pd_{\Lambda}\,soc\,M \leq 1\}$.

\end{thm}

\begin{proof}

Assume pd$_{\Lambda}\,M\leq 1$, then there exists a short exact sequence $0 \rightarrow P_{1} \rightarrow P_{0} \rightarrow M \rightarrow 0$ where $P_{0}$ and $P_{1}$ are projective $\Lambda$-modules. If pd$_{\Lambda}{\rm soc}\,M = 2$, there exists a simple $\Lambda$-module $S\in\, $add$\,{\rm soc}\,M$ such that pd$_{\Lambda}\,S=2$. Let $\varphi:S\rightarrow M$ be the inclusion, then we have the following commutative diagram with exact rows.

\[ \xymatrix{
0\ar[r] &P_{1}\ar@{=}[d] \ar[r] & X \ar[d]^h \ar[r] & S \ar[d]^\varphi \ar[r] & 0 \\
0\ar[r] &P_{1}           \ar[r] & P_{0}      \ar[r] & M                \ar[r] & 0  }
\]

Since $\varphi$ is an injection, by Snake lemma, $h$ is also injective. According to Lemma 2.2, we get Ext$^{1}_{\Lambda}(S,P_{1})=0$. So the higher short exact sequence splits and $X$$=P_{1}\oplus S$. It follows that $S$ is contained in the socle of $P_{0}$, which contradicts the fact that pd$_{\Lambda}\, {\rm soc}\,P_{0}\leq 1$.

Conversely, assume pd$_{\Lambda}\,{\rm soc}\,M\leq1$. Let $I(M)$ be the injective envelope of $M$, then we get a short  exact sequence $0 \rightarrow M \rightarrow I(M) \rightarrow I(M)/M \rightarrow 0$. Note that soc\,$M=$soc\,$I(M)$, we get pd$_{\Lambda}\,$soc$\,I(M)\leq1$. By Theorem 3.2, $I(M)$ is projective. If $I(M)/M$ is a projective $\Lambda$-module, $M$ is also projective. Otherwise, pd$_{\Lambda}\,M=$pd$_{\Lambda}\,I(M)/M-1\leq 1$.

\end{proof}

The above theorem also implies that for a $\Lambda$-module $M$, pd$_{\Lambda}\,M=2$ if and only if pd$_{\Lambda}\,{\rm soc}\,M=2$. This shows that the projective dimension of $M$ is completely determined by the projective dimension of its socle. By Proposition 2.1, the projective dimensions of all simple $\Lambda$-modules are clear, then we can get the projective dimension of any module over the Auslander algebra $\Lambda$.
\section{Characterizations of Auslander algebras}

Auslander algebras are characterized by the properties that global dimension is at most two and dominant dimension is at least two. In [5], Crawley-Boevey and Sauter give a new characterization of the Auslander algebra $\Lambda$, that is gl.d\,$\Lambda\leq 2$ and there exists a tilting $\Lambda$-module which is generated and cogenerated by projective-injective $\Lambda$-modules. In this section, we also give a new characterization of Auslander algebras.

Theorem 3.2 and Theorem 3.3 state two properties of the Auslander algebra $\Lambda$, that is  $\mathcal{C}(\Lambda)$$=\{ I \in { \rm mod}$-$\Lambda\ |\  I\ {\rm is\ injective \ and \ pd}_{\Lambda}\,{\rm soc}\,I\leq 1\}$ and $\mathcal{P}^{1}(\Lambda)=\{M\in { \rm mod}$-$\Lambda\ | \ {\rm pd}_{\Lambda}\,{\rm soc}\,M \leq 1\}$. Then it is natural to consider whether a finite dimensional algebra of global dimension at most two is an Auslander algebra provided it satisfies these properties. We prove the following result.

\begin{thm}

Let $\Lambda$ be a finite dimensional $k$-algebra. Then $\Lambda$ is an Auslander algebra if and only if its global dimension gl.d\,$\Lambda\leq 2$ and $\mathcal{C}(\Lambda)$$=\{ I \in mod$-$\Lambda\ |\  I\ is\ injective \ and \\ pd_{\Lambda}\,soc\,I \leq 1\}$.

\end{thm}

\begin{proof}
If $\Lambda$ is an Auslander algebra, then gl.d\,$\Lambda\leq 2$ and by Theorem 3.2, we have $\mathcal{C}(\Lambda)$$=\{ I \in { \rm mod}$-$\Lambda\ |\  I\ {\rm is\ injective \ and \ pd}_{\Lambda}\,{\rm soc}\,I\leq 1\}$.

Now assume $\Lambda$ is a finite dimensional $k$-algebra with gl.d\,$\Lambda\leq 2$ and  $\mathcal{C}(\Lambda)$$=\{ I \in { \rm mod}$-$\Lambda\ |\  I\ {\rm is\ injective \ and \ pd}_{\Lambda}\,{ \rm soc}\,I\leq 1\}$. We only need to show that the dominant dimension of $\Lambda$ is at least two.

If gl.d\,$\Lambda\leq 1$, then all injective $\Lambda$-modules are projective and $\Lambda$ is a self-injective algebra. For a $\Lambda$-module $M$, there exists a short exact sequence $0 \rightarrow P_{1} \rightarrow P_{0} \rightarrow M \rightarrow 0$ where $P_{0}$ and $P_{1}$ are projective $\Lambda$-modules. Since $P_{1}$ is also injective, this short exact sequence splits and $M$ is projective. It follows that $\Lambda$ is a semi-simple algebra. Thus $\Lambda$ is an Auslander algebra.

If gl.d\,$\Lambda=2$, by Lemma 2.4, we have Id\,$_{\Lambda}\,\Lambda=2$. Let $0 \rightarrow \Lambda \xrightarrow{f} I_{0} \rightarrow I_{1} \rightarrow I_{2} \rightarrow 0 $ be a minimal injective resolution of $\Lambda$. Since $I_{0}$ is the injective envelope of $\Lambda$, we have soc\,$I_{0}=$soc\,$\Lambda$. According to Lemma 2.3, pd$_{\Lambda}\,{\rm soc}\,I_{0}=$pd$_{\Lambda}\,{\rm soc}\,\Lambda\leq 1$, then $I_{0}$ is projective. Consider the short exact sequence $0 \rightarrow \Lambda \xrightarrow{f} I_{0} \rightarrow {\rm Im}\,f \rightarrow 0 $. We claims that pd$_{\Lambda}\, {\rm soc}\, {\rm Im}\,f\leq 1$.

Otherwise,  if pd$_{\Lambda}\, {\rm soc}\,{\rm Im}\,f=2$,  let $g: {\rm soc}\,{\rm Im}\,f \rightarrow {\rm Im}\,f$ be the inclusion. Then we get the following commutative diagram with exact rows.

\[ \xymatrix{
0\ar[r] &\Lambda \ar@{=}[d] \ar[r] & E \ar[d]^h \ar[r] & {\rm soc}\,{\rm Im}\,f \ar[d]^g \ar[r] & 0 \\
0\ar[r] &\Lambda            \ar[r] & I_{0}      \ar[r] &      {\rm Im}\,f          \ar[r] & 0  }
\]

Since $g$ is injective, by Snake lemma, $h$ is also an injection. Note that pd$_{\Lambda}\, {\rm soc}\,{\rm Im}\,f=2 \neq 1+$pd$_{\Lambda}\,\Lambda$, by Proposition 2.5, we have that  pd$_{\Lambda}\,E=max($pd$_{\Lambda}\,\Lambda, $pd$_{\Lambda}\,{\rm soc}\,{\rm Im}\,f)=2$. Consider the short exact sequence $0 \rightarrow E \rightarrow I_{0} \rightarrow  Y \rightarrow 0$. Since $I_{0}$ is projective, we have pd$_{\Lambda}\,Y=$pd$_{\Lambda}\,E+1=3$, a contradiction.

Thus pd$_{\Lambda}\, {\rm soc}\,{\rm Im}\,f\leq 1$ and $I_{1}$ is projective since $I_{1}$ is the injective envelope of Im\,$f$ and soc$\,I_{1}= $soc$\,{\rm Im}\,f$. Then the dominant dimension of $\Lambda$ is at most two and this completes our proof.

\end{proof}

We should mention that a finite dimensional $k$-algebra $\Lambda$ with gl.d\,$\Lambda\leq 2$ and $\mathcal{P}^{1}(\Lambda)=\{M\in { \rm mod}$-$\Lambda\ | \ {\rm pd}_{\Lambda}\,{\rm soc}\,M \leq 1\}$ is not necessarily an Auslander algebra. The following is a counter-example.

{\bf Example.} Let $\Lambda=kQ/\langle\beta\alpha\rangle$ be a finite dimensional $k$-algebra where $Q$ is the quiver $ 1 \xrightarrow{\alpha} 2 \xrightarrow{\beta} 3 \xleftarrow{\gamma} 4$. We have pd$_{\Lambda}\,S_{1}=2$ and pd$_{\Lambda}\,S_{i}\leq 1$ for $i=2,3,4$. Moreover, $S_{1}$ is the only indecomposable $\Lambda$-module whose projective dimension is two. Then gl.d\,$\Lambda=2$ and it is easy to know $\mathcal{P}^{1}(\Lambda)=\{M\in { \rm mod}$-$\Lambda\ | \ {\rm pd}_{\Lambda}\,{\rm soc}\,M \leq 1\}$. However, the indecomposable injective $\Lambda$-module $S_{4}$  is not projective and $\Lambda$ is not an Auslander algebra since its dominant dimension is zero.


\begin{thebibliography}{100}

\bibitem{ARS} \ M.Auslander, I.Reiten, S.O.Smal$\phi$,  Representation Theory of Artin Algebras.  Cambridge Studies in Advanced Math, 36. Cambridge University Press, Cambrideg, 1995.

\bibitem{IY}\ O.Iyama, Auslander correspondence, Advances in Mathematics, 210 (2007),51-82.

\bibitem{DR}\ V.Dlab, C.M.Ringel, Auslander algebras as quasi-hereditary algebras, Journal of the London Mathematical Society, (2) 39 (1989),457-466.

\bibitem{GLS}\ C.Geiss, B.Leclerc, J.Schr$\ddot{o}$er, Auslander algebras and initial seeds for cluster algebras,  Journal of the London Mathematical Society, (2) 75 (2007), 718-740

\bibitem{CBS}\ W.Crawley-Boevey, J.Sauter, On quiver Grassmannians and orbit closures for representation-finite algebras, Mathematische Zeitschrift, arXiv:1509.03460v1.

\bibitem{IZ}\ O.Iyama, X.Zhang, Classifying $\tau$-tilting modules over the auslander algbera of $K[x]/$($x^{n}$), arXiv:1602.05037v1.

\bibitem{OE}\ \"{O}.Eir\'{\i}ksson, From submodule categories to the stable auslander algebra, arXiv:1607.08504v1.

\bibitem{ASS}\ I.Assem , D.Simson , A.Skowronski, Elements of the representation theory of associative algebras. Vol.1. Techniques of representation theory.  London Mathematical Society Student Texts, 65. Cambridge University Press, Cambridge, 2006.

\bibitem{IR}\ O.Iyama, I.Reiten, 2-Auslander algebras associated with reduced words in coxeter group, International Mathematics Research Notices, 8(2011), 1782-1803.

\end{thebibliography}
\end{document}